\documentclass[11pt]{amsart}
\usepackage{amssymb}
\usepackage{amsmath}
\usepackage{comment}
\usepackage{latexsym}
\usepackage{amsfonts}
\usepackage{graphicx}
\usepackage[utf8]{inputenc}
\usepackage[bookmarks=true,pdfstartview=FitH, pdfborder={0 0 0},
colorlinks=true,citecolor=blue, linkcolor=blue, hypertexnames=false]{hyperref}
\newtheorem{theorem}{Theorem}

\newtheorem{corollary}[theorem]{Corollary}

\newtheorem{lemma}[theorem]{Lemma}

\newtheorem{proposition}[theorem]{Proposition}
\newtheorem{remark}[theorem]{Remark}

\newcommand{\M}{\mathcal{M}}
\newcommand{\pd}{\partial}
\newcommand{\Ra}{\rightarrow}
\newcommand{\RR}{\mathbb{R}}
\newcommand{\ZZ}{\mathbb{Z}}

\protected\def\vts{%
  \ifmmode
    \mskip0.5\thinmuskip
  \else
    \ifhmode
      \kern0.08334em
    \fi
  \fi
}

\input{tcilatex}

\begin{document}
\title[Curvature of $4$-dimensional GRS singularity models]{Curvature
growth of
some $4$-dimensional gradient Ricci soliton singularity models}
\author[B. Chow]{Bennett Chow$^{\vts\text{a}}$}
\author[M. Freedman]{Michael Freedman$^{\vts\text{b}}$}
\author[H. Shin]{Henry Shin$^{\vts\text{a}}$}
\author[Y. Zhang]{Yongjia Zhang$^{\vts\text{a}}$}
\date{}

\address{\footnotesize{$^{\text{a}\vts}$\emph{Department of Mathematics, University of California, San Diego, California 92093, USA.}} }

\address{\footnotesize{$^{\text{b}\vts}$\emph{Station Q, Microsoft Research, Santa Barbara, California 93106, USA and Department of Mathematics, University of California, Santa Barbara, California 93106, USA.}} }

\keywords{Ricci flow, gradient Ricci soliton, singularity model, Ricci flat, asymptotically locally Euclidean manifold.}

\begin{abstract}
In this note
we discuss estimates for the curvature of 4-dimensional gradient Ricci soliton singularity models by applying Perelman's point selection, a fundamental result of Cheeger and Naber, and topological lemmas.
\end{abstract}

\maketitle

\section{Introduction}
\subsection{Definition of a singularity model}
For a
finite-time
singular solution to the Ricci flow on a closed oriented manifold $(\mathcal{M}^{n},g(t))$, $t\in [0,T)$, $T<\infty$, we have $\sup_{M\times [0,T)} |\operatorname{Rm}|=\infty$.
An associated \emph{singularity
model} $(\mathcal{M}_{\infty}^{n},g_{\infty}(t))$, $t\in(-\infty,0]$, is a
complete ancient solution which is a limit of pointed rescalings.
More precisely,
there exists a sequence of space-time points $(x_{i},t_{i})$ in $\mathcal{M}\times [0,T)$ with $K_{i}\doteqdot
\left\vert \operatorname{Rm}\right\vert (x_{i},t_{i})\rightarrow\infty$ such
that the sequence of pointed solutions $(\mathcal{M},g_{i}(t),(x_{i},0))$, where $g_{i}%
(t)=K_{i}g(K_{i}^{-1}t+t_{i})$ and $t\in [-K_i t_i ,0]$, converges in the $C^{\infty}$ pointed
Cheeger--Gromov sense to the complete ancient solution $(\mathcal{M}_{\infty},g_{\infty}(t),(x_{\infty},0))$, $t\in (-\infty ,0]$, for
some $x_{\infty}\in\mathcal{M}_{\infty}$. Note that $g_{\infty}(t)$ is not assumed to have bounded curvature on each time slice.

A folklore conjecture is that any singularity model must have bounded curvature. In dimension $3$, this is true by the work of Perelman \cite{P1}.
Observe that it is not obvious that singularity models are necessarily of finite (topological) type. Neither is it obvious that singularity models are embeddable in the compact manifold from which they arise. However, if a singularity model has finite type, then it is embeddable in the original compact manifold and thus has an orientation induced by the embedding.

\subsection{Classification of $3$-dimensional singularity models}
There is now
a complete classification of 3-dimensional singularity models.
We shall use GRS as an abbreviation for gradient Ricci soliton.
Firstly, $3$-dimensional noncollapsed shrinking GRS with bounded curvature have been classified by Hamilton \cite{H95b} and by Perelman \cite{P2}, who proved nonexistence in the noncompact positive sectional curvature case.
Naber \cite{Naber} showed that shrinking GRS with bounded curvature must be noncollapsed.
It was shown that  $3$-dimensional shrinking GRS must have bounded curvature via the works of Cao, Chen, and Zhu
\cite{CCZ}
and Ni and Wallach \cite{NW}, with related works by Chen \cite{Chen} and Petersen and Wylie \cite{PW}. Secondly,
Brendle \cite{Brendle1} proved the assertion by Perelman that the only $3$-dimensional nonflat noncollapsed steady GRS is the rotationally symmetric Bryant soliton.
Thirdly, the works of Hamilton \cite{H95b}, Perelman \cite{P1}, and Brendle \cite{Brendle2}, prove that any $3$-dimensional singularity model must be either a shrinking GRS or a steady GRS. In particular, Brendle \cite{Brendle2} proved Perelman's conjecture that any $3$-dimensional noncollapsed ancient solution with positive bounded sectional curvature must be a steady GRS. See Bamler and Kleiner \cite{BK3} for a later, alternative
treatment
related to their work on strong stability of $3$-dimensional Ricci flow and their proof of the generalized Smale conjecture \cite{BK1,BK2}.
Finally, one
obtains
from Hamilton, Perelman, and Brendle's results that the possible $3$-dimensional oriented singularity models are classified as: a shrinking spherical space form $S^3/\Gamma$, a round cylinder $S^2 \times \RR$ or its $\ZZ_2$-quotient, or the Bryant soliton.

There has been much progress on the understanding of higher-dimensional shrinking GRS, largely due to the works of Munteanu and Wang (see e.g. \cite{MunteanuWang5,MunteanuWang6} and the references therein), with the strongest results in dimension $4$. Wylie \cite{Wylie} proved that the fundamental group is finite. Kotschwar and Wang proved an important uniqueness result \cite{KW}; see also \cite{Kot,KW2}. Haslhofer and M\"uller \cite{HM1,HM2} have proved precompactness theorems.
Y.\,Li and B.\,Wang \cite{LiWang} proved an improved no local collapsing theorem for shrinking GRS.

For recent progress on steady GRS, see Appleton \cite{Appleton1}, Brendle \cite{Brendle2014},
Cao and Cui \cite{CaoCui}, Deng and Zhu \cite{DZ2,DZ1,DZ},
Deruelle \cite{Deruelle},
Munteanu and Sesum \cite{MunSes}, and Munteanu, Sung, and Wang \cite{MSW}, Munteanu and Wang \cite{MunWan1,MunWan2,MunWan3}, and the references therein.
For the classification of $4$-dimensional shrinking GRS with positive isotropic curvature, see Li, Ni, and Wang \cite{LNW}.
Very recently, Chen, Deruelle, and Sun \cite{CDS} proved that any $4$-real-dimensional K\"ahler shrinking GRS with scalar curvature tending to zero at infinity must be the Feldman--Ilmanen--Knopf shrinker \cite{FIK}.

In the case where the potential function is constant, a steady GRS is
a Ricci-flat manifold. Assuming also that
this
manifold
is a singularity model, by Perelman's no local collapsing theorem it is necessarily $\kappa$-noncollapsed on all scales. In particular, it has Euclidean volume growth. In dimension 4, by a fundamental result of Cheeger and Naber \cite{CN}, it must be
a Ricci-flat asymptotically locally Euclidean (ALE) manifold. We remark that singularity models of 4-dimensional singular solutions with bounded scalar curvature have been shown to be such ALE manifolds by Bamler and Zhang \cite{BamlerZhang1} and Simon \cite{SimonMiles}.

\subsection{Main results}
In this paper we study
the curvature growth of
4-dimensional
GRS singularity models.

\begin{theorem}
Any $4$-dimensional steady GRS
singularity model $(\mathcal{M},g,f)$ must have bounded curvature, that is, there exists a constant $C$ depending on the GRS such that $|\operatorname{Rm}|\leq C$ on $\M$.
\end{theorem}

For similar reasons, we have:

\begin{theorem}
Any $4$-dimensional shrinking GRS
singularity model must have curvature which grows at most quadratically. That is, for any $o\in \mathcal{M}$ there exists a constant $C$ depending only on the GRS and $o$ such that $|\operatorname{Rm}|(x)\leq C(d(x,o)+1)^2$ for all $x\in\M$.
\end{theorem}

The paper is organized as follows. In Section \ref{Section 2} we
first
recall Perelman's point selection method on Riemannian manifolds, which is generally used to obtain limits. We then discuss under what conditions the local derivative of curvature estimates of Shi yield instantaneous estimates for GRS. Next, we prove the main Theorems 1 and 2 modulo the topological lemmas proved in Section 3.

For a survey of $4$-dimensional Ricci flow, see \cite{CFGH} by Gompf, Hillman, and two of the authors.

\section{Curvature estimates for GRS singularity models}\label{Section 2}
\subsection{Point selection}
The following is Perelman's point selection method; see \cite{P1}. Since the method is crucial to our results, we include its proof for the sake of completeness.

\begin{lemma}
Let $(\mathcal{M}^{n},g)$ be a complete Riemannian manifold. For any $%
y_{0}\in \mathcal{M}$ \emph{(}let $P_{0}\doteqdot \left\vert \func{Rm}%
\right\vert (y_{0})$\emph{)} and $A_{0}\in \mathbb{R}^{+}$ there exists $%
x_{0}\in B_{2A_{0}P_{0}^{-1/2}}(y_{0})$ such that $Q_{0}\doteqdot \left\vert
\func{Rm}\right\vert (x_{0})\geq P_{0}$ and%
\begin{equation}
\left\vert \func{Rm}\right\vert \leq 4Q_{0}\quad \text{in }%
B_{A_{0}Q_{0}^{-1/2}}(x_{0}).
\label{Curvature control at arbitrary distance}
\end{equation}%
For example, we may choose $A_{0}=\frac{1}{3}P_{0}^{1/2}$, in which case we
also have $A_{0}Q_{0}^{-1/2}\leq \frac{1}{3}$, so that $x_{0}\in
B_{2/3}(y_{0})$ and $B_{A_{0}Q_{0}^{-1/2}}(x_{0})\subset B_{1}(y_{0})$.
\end{lemma}

\begin{proof}
If (\ref{Curvature control at arbitrary distance}) holds for $x_{0}=y_{0}$,
then we are done. So suppose (\ref{Curvature control at arbitrary distance})
does not hold for $x_{0}=y_{0}$. Let $O_{0}\doteqdot \left\vert \func{Rm}%
\right\vert \left( y_{0}\right) =P_{0}$. Since (\ref{Curvature control at
arbitrary distance}) does not hold for $x_{0}=y_{0}$, there exists $y_{1}\in
B_{A_{0}O_{0}^{-1/2}}(y_{0})$ such that%
\begin{equation*}
O_{1}\doteqdot \left\vert \func{Rm}\right\vert \left( y_{1}\right) >4O_{0}.
\end{equation*}%
By induction, suppose that $y_{0},\ldots ,y_{j}$ have been chosen such that (%
\ref{Curvature control at arbitrary distance}) does not hold for $x_{0}=y_{k}
$ for all $0\leq k\leq j-1$ and with $O_{k}\doteqdot \left\vert \func{Rm}%
\right\vert \left( y_{k}\right) $ for $0\leq k\leq j$ we have $y_{k+1}\in
B_{A_{0}O_{k}^{-1/2}}(y_{k})$ satisfies $O_{k+1}>4O_{k}$ for $0\leq k\leq j-1
$. Then, for $0\leq k\leq j$ we have%
\begin{equation}
O_{k}\geq 4^{k}O_{0}=4^{k}P_{0}  \label{Ok 4 kay minus 1 O1}
\end{equation}%
and%
\begin{align}
d\left( y_{k},y_{0}\right) & =\sum_{\ell =0}^{k-1}d\left( y_{\ell
+1},y_{\ell }\right)   \label{Distance from zk to y0 less than 2 thirds} \\
& <\sum_{\ell =0}^{k-1}A_{0}O_{\ell }^{-1/2}  \notag \\
& \leq A_{0}\sum_{\ell =0}^{k-1}O_{0}^{-1/2}2^{-\ell }  \notag \\
& <2A_{0}O_{0}^{-1/2}.  \notag
\end{align}%
Hence $y_{k}\subset B_{2A_{0}O_{k}^{-1/2}}(y_{0})$ and $%
B_{A_{0}O_{k}^{-1/2}}(y_{k})\subset B_{3A_{0}O_{k}^{-1/2}}(y_{0})$ for $%
0\leq k\leq j$.

For the sequence $y_{0},y_{1},\ldots $ there exists a first $j_{0}$ for
which (\ref{Curvature control at arbitrary distance}) holds for $%
x_{0}=y_{j_{0}}$, for otherwise we would have an infinite sequence of points
$\{y_{k}\}_{k=0}^{\infty }$ for which $\left\vert \func{Rm}\right\vert
\left( y_{k}\right) \geq 4^{k}P_{0}$ and $y_{k}\in
B_{2A_{0}O_{0}^{-1/2}}(y_{0})$ for all $k\geq 0$, a contradiction. The lemma
follows from taking $x_{0}=y_{j_{0}}$ since $\left\vert \func{Rm}\right\vert
\left( y_{j_{0}}\right) \geq 4^{j_{0}}P_{0}\geq P_{0}$ and since (\ref%
{Distance from zk to y0 less than 2 thirds}) implies $d\left(
y_{j_{0}},y_{0}\right) <2A_{0}P_{0}^{-1/2}$.
\end{proof}

To summarize, given any $A_{0}>0$ and point $y_{0}$, Perelman's point
selection method finds a nearby point $x_{0}$ such that $\left\vert \func{Rm}%
\right\vert $ in the ball centered at $x_{0}$ of scaled radius $A_{0}$ is
bounded by $4$ times its value at $x_{0}$. This is effective since $A_{0}$
is arbitrary. So, given a sequence $\{y_{i}\}$ with $\left\vert \func{Rm}%
\right\vert (y_{i})\rightarrow \infty $, we may choose $\left\{
A_{i}\right\} $ so that $A_{i}\rightarrow \infty $. In particular, as an
immediate consequence of the lemma, we have:

\begin{proposition}
\label{Perelman point selection for Rm tending to infinity}Let $(\mathcal{M}%
^{n},g)$ be a complete Riemannian manifold. For any sequence $%
\{y_{i}\}_{i=1}^{\infty}$ in $\mathcal{M}$ with $P_{i}\doteqdot\left\vert
\func{Rm}\right\vert (y_{i})\rightarrow\infty$ there exists a sequence $%
\{x_{i}\}_{i=1}^{\infty}$ in $\mathcal{M}$ such that for each $i\geq1$ we
have $x_{i}\in B_{2/3}(y_{i})$, $Q_{i}\doteqdot\left\vert \func{Rm}%
\right\vert (x_{i})\geq P_{i}$, and%
\begin{equation*}
\left\vert \func{Rm}\right\vert \leq4Q_{i}\quad\text{in }%
B_{A_{i}Q_{i}^{-1/2}}(x_{i}),\quad\text{where }A_{i}=\frac{1}{3}%
P_{i}^{1/2}\rightarrow\infty.
\end{equation*}
\end{proposition}

\subsection{Instantaneous local derivative estimates}
We have the following (instantaneous in time) local derivative estimates. This is useful for rescalings about points where $\left\vert \operatorname{Rm}%
\right\vert $ is bounded below by a positive constant.

\begin{lemma}
\label{Local Deriv Est For Steady GRS}
Let $(\mathcal{M}^{n},g,f)$ be a complete steady or shrinking GRS.
Suppose that $p\in \M$, $r>0$, and $C$
are such that in $B_{2r}(p)$ we have
$\left\vert \limfunc{Rm}\right\vert \leq Cr^{-2}$
and $|\nabla f|\leq r^{-1}$.
Then $\left\vert \nabla ^{m}\limfunc{Rm}\right\vert \leq C_{m}r^{-2-m}$ in $B_{r}(p)$, where $C_{m}$ depends only on $m$, $C$, and $n$.
\end{lemma}

\begin{proof}
Let $(\mathcal{M},g(t),f(t))$, $\lambda t<1$, be the canonical form associated to the steady
GRS; see e.g. Theorem 4.1 in \cite{CLN}. Here $\lambda=0$ in the steady case and $\lambda=1$  in the shrinking case, so that $\lambda \geq 0$.
By hypothesis, $\left\vert \limfunc{Rm}\right\vert (x,0)\leq Cr^{-2}$ for
$x\in B_{2r}^{g(0)}(p)$.
By the definition of the canonical form,
\[
\frac{\partial }{\partial t}\varphi
_{t}\left( x\right) =\frac{1}{1-\lambda t}\left( \nabla _{g(0)}f(0)\right) \left( \varphi
_{t}\left( x\right) \right) ,
\]
$g(t) = (1-\lambda t) g(0) $, and $f(t)=f(0) \circ \varphi
_{t}$ satisfy
\begin{equation*}
\frac{\partial}{\partial t}g(t)=-2\operatorname{Rc}
= 2 \nabla^2_{g(t)} f(t) - \frac{\lambda}{1-\lambda t} g(t) .
\end{equation*}
That is, $(\M,g(t))$ is a solution to the Ricci flow evolving by scaling and by the pullback by a $1$-parameter family of diffeomorphisms.
Hence,
by the assumption that $|\nabla f(0)|_{g(0)}\leq r^{-1}$ in $B^{g(0)}_{2r}(p)$, we have the inequality
\[
\left\vert
\frac{\partial }{\partial t}\varphi _{t}\left( x\right) \right\vert
_{g(0)}\leq r^{-1}
\]
whenever $\varphi_t(x)\in B^{g(0)}_{2r}(p)$ and $t\leq 0$.  Hence, if $x\in B_{3r/2}^{g(0)}(p)$ and $t\in \lbrack -\frac{%
r^{2}}{2},0]$, then%
\begin{equation*}
d_{g(0)}(\varphi _{t}\left( x\right) ,x)\leq \int_{t}^{0}\left\vert \frac{%
\partial }{\partial t}\varphi _{\bar{t}}\left( x\right) \right\vert
_{g(0)}d\,\bar{t}\leq \frac{r}{2}.
\end{equation*}%
Thus $\varphi _{t}\left( x\right) \in B_{2r}^{g(0)}(p)$, so
that $\left\vert \limfunc{Rm}\right\vert (x,t)=\left\vert \limfunc{Rm}%
\right\vert (\varphi _{t}(x),0)\leq Cr^{-2}$. By Shi's local derivative of
curvature estimates, $\left\vert \nabla ^{m}\limfunc{Rm}\right\vert
(x,0)\leq C_{m}r^{-2-m}$ for $x\in B_{r}^{g(0)}(p)$.
\end{proof}

The condition $|\nabla f|\leq r^{-1}$ puts some restriction on the scales $r$ for which we may apply this lemma. For instance, on a steady GRS it is only known that $|\nabla f|^2=1-R\leq 1$; hence this results works naturally for $0<r\leq 1$ (or less than any fixed constant). On the other hand, on a shrinking GRS, since we only have $|\nabla f|\leq\sqrt{f-R}\leq\frac{1}{2}d(o,x)+C$, this lemma works naturally for $0<r\lesssim d^{-1}(o,x)$. In particular, for the steady case, we have:

\begin{corollary}
Let $(\mathcal{M}^{n},g,f)$ be a complete steady GRS with $\left\vert \limfunc{Rm}\right\vert \leq Cr^{-2}$ in $B_{2r}(p)$, where $r\in (0,1]$.
Then $\left\vert \nabla ^{m}\limfunc{Rm}\right\vert \leq C_{m}r^{-2-m}$ in $B_{r}(p)$, where $C_{m}$ depends only on $m$, $C$, and $n$.
\end{corollary}

\subsection{Four-dimensional Ricci-flat ALE manifolds}\label{sub ALE def}

We say that a complete noncompact oriented Riemannian
$4$-manifold $(\mathcal{M}^4,g)$ is \emph{asymptotically locally Euclidean} (\emph{ALE})
if there exists $\tau > 0$, a compact subset $K$, a finite subgroup $\Gamma$ of $SO(4)$ acting freely on $S^3$, and an orientation-preserving diffeomorphism $\Phi : (\RR^4 - B_C(0))/\Gamma \Ra \M - K$ for some $C$ such that the pullback $\tilde{h}$ to the cover $\RR^4 - B_C(0)$ of $h = \Phi^*g$ satisfies $\pd_I (\tilde{h}_{jk} - \delta_{jk}) = O (|x|^{-\tau -|I|})$ for each multi-index $I=(i_1, i_2, i_3 ,i_4)$.
By Bando, Kasue, and Nakajima \cite[Theorem 1.5]{BKN}, if an ALE manifold is Ricci flat, then there exists $\Phi$ with $\tau = 4$.

It is conjectured that any simply-connected Ricci-flat ALE $4$-manifold must be hyperk\"ahler, the latter of which has been classified by Kronheimer \cite{K1, K2}. For progress in this direction, see Lock and Viaclovsky \cite{LV}.
Appleton \cite{Appleton2} has shown that the Eguchi--Hanson space, which is a hyperk\"ahler ALE, can occur as a singularity model of $4$-dimensional Ricci flow.

\subsection{Steady GRS singularity models}
Our first main result is the following.

\begin{proposition}
\label{steady models have bounded curvature}
If $(\mathcal{M}^{4},g,f)$ is a
steady GRS which is also a singularity model, then $\left\vert \func{Rm}%
\right\vert $ is bounded.
\end{proposition}

\begin{proof}
Suppose $\left\vert \func{Rm}\right\vert $ is not bounded. Then there exists
$\{y_{i}\}_{i=1}^{\infty}$ in $\mathcal{M}$ with $P_{i}\doteqdot \left\vert
\func{Rm}\right\vert (y_{i})\rightarrow\infty$. By Proposition \ref{Perelman
point selection for Rm tending to infinity}, there exists a sequence $%
\{x_{i}\}_{i=1}^{\infty}$ in $\mathcal{M}$ such that for each $i\geq1$ we
have $x_{i}\in B_{2/3}(y_{i})$, $Q_{i}\doteqdot\left\vert \func{Rm}%
\right\vert (x_{i})\geq P_{i}\geq100$, and%
\begin{equation}
\left\vert \func{Rm}\right\vert \leq4Q_{i}\quad\text{in }%
B_{A_{i}Q_{i}^{-1/2}}(x_{i}),\quad\text{where }A_{i}=\frac{1}{3}%
P_{i}^{1/2}\rightarrow\infty.
\label{Curvature control on larger and larger balls Ai}
\end{equation}
Since $(\mathcal{M}^{4},g,f)$ is a singularity model, it is $\kappa $%
-noncollapsed on all scales for some $\kappa>0$. Hence (\ref{Curvature
control on larger and larger balls Ai}) implies%
\begin{equation*}
\func{Vol}B_{Q_{i}^{-1/2}}(x_{i})\geq\kappa Q_{i}^{-2}\quad\text{for all }%
i\geq1.
\end{equation*}
Let $(\mathcal{M}^{4},g(t),f(t))$, $t\in(-\infty,\infty)$, be the associated
canonical form, where $g(t)=\varphi_{t}^{\ast}g$ and $f(t)=f\circ\varphi_{t}$%
. Since $\frac{\partial}{\partial t}\varphi_{t}\left( x\right) =\left(
\nabla_{g(0)}f(0)\right) \left( \varphi_{t}\left( x\right) \right) $, we
have $\left\vert \frac{\partial}{\partial t}\varphi_{t}\left( x\right)
\right\vert _{g(0)}\leq1$. Hence, if $x\in
B_{(A_{i}-1)Q_{i}^{-1/2}}^{g(0)}(x_{i})$ and $t\in\lbrack-Q_{i}^{-1},0]$,
then%
\begin{equation*}
d_{g(0)}(\varphi_{t}\left( x\right) ,x)\leq\left\vert t\right\vert \leq
Q_{i}^{-1}\leq Q_{i}^{-1/2},
\end{equation*}
so that $\varphi_{t}\left( x\right) \in B_{A_{i}Q_{i}^{-1/2}}^{g(0)}(x_{i})$
and hence $\left\vert \limfunc{Rm}\right\vert (x,t)=\left\vert \limfunc{Rm}%
\right\vert (\varphi_{t}(x),0)\leq4Q_{i}$. Hence, by Shi's local derivative
estimates, we have%
\begin{equation*}
\left\vert \nabla^{m}\limfunc{Rm}\right\vert \leq C_{m}Q_{i}^{1+\frac {m}{2}%
}\quad\text{in }B_{\frac{1}{2}A_{i}Q_{i}^{-1/2}}^{g(0)}(x_{i})\times\lbrack-%
\frac{1}{2}Q_{i}^{-1},0].
\end{equation*}

By the above and by Hamilton's Cheeger--Gromov compactness theorem, there
exists a subsequence such that $(\mathcal{M}^{4},Q_{i}g(Q_{i}^{-1}t),x_{i})$
converges in the $C^{\infty}$ pointed Cheeger--Gromov sense to a complete
solution to the Ricci flow $(\mathcal{M}_{\infty}^{4},g_{\infty}(t),x_{%
\infty })$, $t\in\lbrack-\frac{1}{2},0]$, with $\left\vert \func{Rm}%
_{g_{\infty}}\right\vert \leq4$ and $\left\vert \func{Rm}_{g_{\infty}}\right%
\vert (x_{\infty},0)=1$.
In particular, $g_{\infty}(t)$ is nonflat.
Since $(\mathcal{M}^{4},g,f)$ is a steady GRS, we
have $0\leq R\leq1$. This and $Q_{i}\rightarrow\infty$ imply that $%
R_{g_{\infty}}\equiv0$, which in turn implies that $\func{Rc}%
_{g_{\infty}}\equiv0$ by the Ricci flow equation $\frac{\partial
R_{g_{\infty }}}{\partial t}=\Delta_{g_{\infty}}R_{g_{\infty}}+2\left\vert
\func{Rc}_{g_{\infty}}\right\vert ^{2}$.

By Perelman's no local collapsing
theorem and since $R_{g_{\infty}(0)}=0$ and we are on a singularity model (a limit of a limit is a limit), there exists $\kappa >0$ such that $\operatorname{Vol}_{g_{\infty}(0)}B^{g_{\infty}(0)}_r(x_{\infty})\geq \kappa r^4$ for all $r>0$
(see e.g. Theorems 28.6 and 28.9 in \cite{Cetc});\footnote{This follows from Perelman's no local collapsing using the scalar curvature and the fact that the rescaling factors tend to infinity.} hence
by definition $g_{\infty}(0)$ has Euclidean volume growth. By Cheeger and Naber's
theorem (Corollary 8.86 of \cite{CN}), $(\mathcal{M}_{\infty},g_{\infty}(0))$ is asymptotically locally Euclidean (ALE).

Since $(\mathcal{M}_{\infty}^{4},g_{\infty}(0))$ is a nonflat Ricci-flat ALE $4$%
-manifold, there exists a finite subgroup $\Gamma$ of $\func{SO}(4)$ such
that the end of $\mathcal{M}_{\infty}$ is diffeomorphic to $(\mathcal{S}%
^{3}/\Gamma)\times(0,1)$.\footnote{A nonflat Ricci-flat ALE $4$-manifold has only one end since otherwise it will split as the
product of a line and a Ricci flat three-manifold, which in turn implies that it is flat.} Let $\mathcal{S}^{3}/\Gamma \subset\mathcal{M}%
_{\infty}$ be embedded so that the noncompact region in $\mathcal{M}%
_{\infty}$ bounded by $\mathcal{S}^{3}/\Gamma$ is diffeomorphic to $(%
\mathcal{S}^{3}/\Gamma)\times\lbrack0,1)$. Let $\Omega$ denote the compact
region in $\mathcal{M}_{\infty}$ bounded by $\mathcal{S}^{3}/\Gamma$. Then $\Omega^{\circ}$ is diffeomorphic to $\mathcal{M}_{\infty}$. By
the definition of Cheeger--Gromov convergence, there exist an exhaustion $%
U_{i}$ of $\mathcal{M}_{\infty}$ and embeddings $\psi_{i}:U_{i}%
\rightarrow B_{\frac{1}{2}A_{i}Q_{i}^{-1/2}}^{g(0)}(x_{i})$. In particular, $%
\psi_{i}(\Omega)$ are embeddings of $\Omega$ with $\psi _{i}(\Omega)\subset
B_{1}^{g(0)}(x_{i})$ for each $i$ sufficiently large. So we may assume that
the $\psi_{i}(\Omega)$ are disjoint from each other in the steady GRS $\mathcal{M}$.

By hypothesis, there exists a solution to the Ricci flow on a closed $4$%
-manifold $(\mathcal{N}^{4},h(t))$, $t\in\lbrack0,T)$, $T<\infty$, for which the steady GRS
$(\mathcal{M},g,f)$ is a singularity model. Hence there exist an
exhaustion $\{V_{j}\}$ of $\mathcal{M}$ and embeddings $%
\phi_{j}:V_{j}\rightarrow\mathcal{N}$. In particular, for any $I\in%
\mathbb{N}$ there exists $j(I)$ such that $V_{j(I)}$ contains $%
\psi_{i}(\Omega)$ for all $1\leq i\leq I$. Then the $\Omega_{i}\doteqdot(%
\phi_{j(I)}\circ\psi _{i})(\Omega)$ are disjoint embeddings into $\mathcal{N}$ for $1\leq i\leq I$.
Since the $\Omega_{i}^{\circ}$ are each diffeomorphic to $\mathcal{M}_{\infty}$ and are pairwise disjoint and since $I$ is arbitrary, we obtain a contradiction to Theorem \ref{CorExcludeInfinite2} below.
\end{proof}

\subsection{Shrinking GRS singularity models}
By a similar argument we can prove the following.

\begin{proposition}
\label{Quad at most growth}
If $(\mathcal{M}^{4},g,f)$ is a shrinking GRS
which is also a singularity model and $o\in \mathcal{M}$, then there exists
a constant $C$ such that%
\begin{equation*}
\left\vert \func{Rm}\right\vert (x)\leq C(d(x,o)+1)^{2}\quad \text{for all }%
x\in \mathcal{M}.
\end{equation*}
\end{proposition}

\begin{proof}
Suppose there exists a sequence of points in the shrinking GRS $\{y_{i}\}_{i=1}^{\infty }$ in $%
\mathcal{M}$ with $P_{i}(d(y_{i},o)+1)^{-2}\rightarrow \infty $, where $%
P_{i}\doteqdot \left\vert \func{Rm}\right\vert (y_{i})$. By a previous
proposition, there exists a sequence $\{x_{i}\}_{i=1}^{\infty }$ in $%
\mathcal{M}$ such that for each $i\geq 1$ we have $x_{i}\in B_{2/3}(y_{i})$,
$Q_{i}\doteqdot \left\vert \func{Rm}\right\vert (x_{i})\geq P_{i}$, and%
\begin{equation*}
\left\vert \func{Rm}\right\vert \leq 4Q_{i}\quad \text{in }%
B_{A_{i}Q_{i}^{-1/2}}(x_{i}),\quad \text{where }A_{i}=\frac{1}{3}%
P_{i}^{1/2}\rightarrow \infty .
\end{equation*}%
Let $(\mathcal{M}^{4},g(t),f(t))$, $t\in (-\infty ,1)$, be the associated
canonical form, where $g(t)=\left( 1-t\right) \varphi _{t}^{\ast }g$ and $%
f(t)=f\circ \varphi _{t}$. Since $\frac{\partial }{\partial t}\varphi
_{t}\left( x\right) =\frac{1}{1-t}\left( \nabla _{g(0)}f(0)\right) \left(
\varphi _{t}\left( x\right) \right) $, we have%
\begin{align*}
\left\vert \frac{\partial }{\partial t}\varphi _{t}\left( x\right)
\right\vert _{g(0)}& \leq \frac{1}{1-t}\left\vert \nabla
_{g(0)}f(0)\right\vert \left( \varphi _{t}\left( x\right) \right)  \\
& \leq \frac{1}{1-t}f^{1/2}\left( \varphi _{t}\left( x\right) ,0\right)  \\
& \leq \frac{1}{1-t}\left( d_{g(0)}(\varphi _{t}\left( x\right) ,o)+C\right)
\\
& \leq d_{g(0)}(\varphi _{t}\left( x\right) ,o)+C
\end{align*}%
for $t\leq 0$ since $|\nabla (f^{1/2}(0))|_{g(0)}\leq 1$ by $R\geq 0$. Hence, if $x\in B_{\frac{3}{4}A_{i}Q_{i}^{-1/2}}^{g(0)}(x_{i})
$ and $t\in \lbrack -Q_{i}^{-1},0]$, then%
\begin{align}
d_{g(0)}(\varphi _{t}\left( x\right) ,x)& \leq \int_{t}^{0}\left\vert \frac{%
\partial }{\partial t}\varphi _{t}\left( x\right) \right\vert _{g(0)}dt
\label{Use this also in the quad case} \\
& \leq \int_{t}^{0}\left( d_{g(0)}(\varphi _{t}\left( x\right) ,o)+C\right)
dt  \notag \\
& \leq CQ_{i}^{-1}\left( d_{g(0)}(x_{i},o)+C\right)   \notag \\
& \ll Q_{i}^{-1/2}  \notag
\end{align}%
since $Q_{i}(d_{g(0)}(x_{i},o)+1)^{-2} \gtrsim P_{i}(d_{g(0)}(y_{i},o)+1)^{-2} \rightarrow \infty $. Thus $\varphi
_{t}\left( x\right) \in B_{A_{i}Q_{i}^{-1/2}}^{g(0)}(x_{i})$ and hence $%
\left\vert \limfunc{Rm}\right\vert (x,t)=\left\vert \limfunc{Rm}\right\vert
(\varphi _{t}(x),0)\leq 4Q_{i}$. Hence we may apply Shi's local derivative
estimates and Hamilton's Cheeger--Gromov compactness theorem to obtain that
there exists a subsequence such that $(\mathcal{M}%
^{4},Q_{i}g(Q_{i}^{-1}t),x_{i})$ converges in the $C^{\infty }$ pointed
Cheeger--Gromov sense to a complete solution to the Ricci flow $(\mathcal{M}%
_{\infty }^{4},g_{\infty }(t),x_{\infty })$, $t\in \lbrack -\frac{1}{2},0]$,
with $\left\vert \func{Rm}_{g_{\infty }}\right\vert \leq 4$ and $\left\vert
\func{Rm}_{g_{\infty }}\right\vert (x_{\infty },0)=1$. Since $(\mathcal{M}%
^{4},g,f)$ is a shrinking GRS, we have $0<R\leq f$. This and $%
Q_{i}f^{-1}(x_{i})\rightarrow \infty $ imply that $R_{g_{\infty }}\equiv 0$,
which in turn implies that $\func{Rc}_{g_{\infty }}\equiv 0$ and so $%
g_{\infty }(0)$ is ALE. The remainder of the proof will be the same as in the
steady GRS case.
\end{proof}

\subsection{Steady GRS models with curvature not decaying}

\begin{proposition}\label{PROP 10}
Let $(\mathcal{M}^{4},g,f)$ be a steady GRS which is also a singularity
model. If there exist a constant $c>0$ and a
sequence $y_{i}\rightarrow \infty $ such that%
\begin{equation*}
\left\vert \func{Rm}\right\vert (y_{i})\geq c\quad \text{%
for all }i,
\end{equation*}%
then
$(\M , g, y_i)$ subconverges
to a non-Ricci-flat steady GRS with bounded curvature.
\end{proposition}
\begin{proof}
Observe that we may assume that $(\mathcal{M}^{4},g)$ is not Ricci flat, since otherwise there does not exist a sequence $\{ y_i\}$ as in the hypothesis by Cheeger and Naber's aforementioned theorem.
Since the singularity model has bounded curvature by Proposition \ref{steady models have bounded curvature}, it follows from Shi's local derivative estimates (or Lemma \ref{Local Deriv Est For Steady GRS}) that the covariant derivative of curvature of each order is uniformly bounded. Consider the sequence $\{(\mathcal{M},g,f_i,y_i)\}_{i=1}^\infty$, where $f_i(x)=f(x)-f(y_i)$. We have $|\nabla f_i|=|\nabla f|=\sqrt{1-R}$ and $|\nabla^2 f|=|\func{Rc}|$ are both uniformly bounded. By applying Shi's Bernstein-type estimates and by the covariant derivatives of curvature bounds, it is not hard to show that the covariant derivatives of $f_i$ of each order are bounded independent of $i$. Taking into account the noncollapsing condition, we may extract a subsequence from $\{(\mathcal{M},g,f_i,y_i)\}_{i=1}^\infty$ which converges in the pointed smooth Cheeger--Gromov sense to $(\mathcal{M}_\infty^{4},g_\infty,f_\infty,y_\infty)$. It then follows from the smooth convergence that
\begin{eqnarray*}
\func{Rc}_\infty+\nabla^2 f_\infty&=&0,
\\
|\nabla f_\infty|^2+R_\infty&=&1.
\end{eqnarray*}
The limit is evidently a steady GRS. It is also nonflat since
\begin{eqnarray*}
c\leq|\func{Rm}|(y_i)\rightarrow |\func{Rm}_\infty|(y_\infty)\geq c>0.
\end{eqnarray*}
To see that it is not Ricci flat, we assume for a contradiction that $g_\infty$ is Ricci flat. Then immediately we have
\begin{eqnarray*}
\nabla^2f_\infty&=&0,
\\
|\nabla f_\infty|^2&=&1.
\end{eqnarray*}
Hence $(\mathcal{M}_\infty,g_\infty)$ splits as the product of a line and a three-dimensional Ricci-flat manifold, which is flat; a contradiction.
\end{proof}

\subsection{Shrinking GRS models with quadratic curvature growth} The following result is proved using similar methods.

\begin{proposition}\label{PROP 9}
Let $(\mathcal{M}^{4},g,f)$ be a shrinking GRS which is also a singularity
model and let $o\in \mathcal{M}$. If there exist a constant $c>0$ and a
sequence $y_{i}\rightarrow \infty $ such that%
\begin{equation*}
\left\vert \func{Rm}\right\vert (y_{i})\geq c(d(y_{i},o)+1)^{2}\quad \text{%
for all }i,
\end{equation*}%
then there exists $x_{i}\rightarrow \infty $ with associated rescalings
limiting to a non-Ricci-flat steady GRS with bounded curvature.
\end{proposition}

\begin{proof}
Let $\left\{ x_{i}\right\} $, $Q_{i}$, and $A_{i}$ be as in Proposition \ref%
{Perelman point selection for Rm tending to infinity}. By (\ref{Use this
also in the quad case}), we have that if $x\in B_{\frac{3}{4}%
A_{i}Q_{i}^{-1/2}}^{g(0)}(x_{i})$ and $t\in \lbrack -Q_{i}^{-1},0]$, then%
\begin{equation*}
d_{g(0)}(\varphi _{t}\left( x\right) ,x)\leq CQ_{i}^{-1}\left(
d_{g(0)}(x_{i},o)+C\right) \leq CQ_{i}^{-1/2}.
\end{equation*}%
Since $A_{i}\rightarrow \infty $, we have $\varphi _{t}\left( x\right) \in
B_{A_{i}Q_{i}^{-1/2}}^{g(0)}(x_{i})$ for $i$ sufficiently large. As in the
proof of Proposition \ref{Quad at most growth}, there exists a subsequence
such that $(\mathcal{M},Q_{i}g(Q_{i}^{-1}t),x_{i})$ converges in the $%
C^{\infty }$ pointed Cheeger--Gromov sense to a complete solution to the
Ricci flow $(\mathcal{M}_{\infty }^{4},g_{\infty }(t),x_{\infty })$, $t\in
\lbrack -\frac{1}{2},0]$, with $\left\vert \func{Rm}_{g_{\infty
}}\right\vert \leq 4$ and $\left\vert \func{Rm}_{g_{\infty }}\right\vert
(x_{\infty },0)=1$. Let $g_{i}=Q_{i}g$ and let $f_{i}(x)=f\left( x\right)
-f\left( x_{i}\right) $. We have%
\begin{equation*}
\func{Rc}_{g_{i}}+\nabla _{g_{i}}^{2}f_{i}=\frac{1}{2}Q_{i}^{-1}g_{i}.
\end{equation*}%
For any $A>0$ and for any $x\in B_{A}^{g_{i}}(x_{i})$, we have%
\begin{eqnarray} \label{Normalization}
\left\vert \nabla f_{i}\right\vert^2 _{g_{i}}+R_{g_i}=\frac{f}{Q_i}\geq\frac{\inf\left\{f(y):y\in B_{AQ_i^{-\frac{1}{2}}}^g(x_i)\right\}}{C(d_g(x_i,o)+1)^2}\geq c>0,
\end{eqnarray}
for all $i$ large enough, where $c>0$ is independent of $A$ and where we have used the lower bound for the potential function of Cao and Zhou \cite{CZ}. Furthermore, we have
\begin{equation*}
\left\vert \nabla f_{i}\right\vert _{g_{i}}(x)=Q_{i}^{-1/2}\left\vert \nabla
f\right\vert (x)\leq Q_{i}^{-1/2}f^{1/2}(x)\leq C,
\end{equation*}%
\begin{equation*}
\left\vert \nabla _{g_{i}}^{2}f_{i}\right\vert _{g_{i}}=\left\vert \func{Rc}%
_{g_{i}}-\frac{1}{2}Q_{i}^{-1}g_{i}\right\vert _{g_{i}}\leq C,
\end{equation*}%
and%
\begin{equation*}
\left\vert \nabla _{g_{i}}^{k}f_{i}\right\vert _{g_{i}}=\left\vert \nabla
_{g_{i}}^{k-2}\func{Rc}_{g_{i}}\right\vert _{g_{i}}\leq C_{k}\quad \text{for
}k\geq 3.
\end{equation*}%
Hence, by taking $i\rightarrow \infty $ we obtain $f_{i}\rightarrow
f_{\infty }$ in $C^{\infty }$ with respect to the Cheeger--Gromov
convergence of $g_{i}$ to $g_{\infty }=g_{\infty }(0)$, where%
\begin{eqnarray*}
\left\vert\nabla_{g_\infty} f_\infty\right\vert^2_{g_\infty}+R_{g_\infty}&=&c_0>0,
\\
\func{Rc}_{g_{\infty }}+\nabla _{g_{\infty }}^{2}f_{\infty }&=&0.
\end{eqnarray*}%
Notice that the left-hand-side of the first equation is constant because the steady GRS equation holds by the second equation; it follows from (\ref{Normalization}) that this constant is positive. According to our point selection procedure, the limit steady GRS $(\mathcal{M}^4_\infty,g_\infty,f_\infty)$ is nonflat and has bounded curvature. To see that it is non-Ricci-flat, we assume the contrary and obtain
\begin{eqnarray*}
\left\vert\nabla_{g_\infty} f_\infty\right\vert^2_{g_\infty}&=&c_0>0,
\\
\nabla _{g_{\infty }}^{2}f_{\infty }&=&0.
\end{eqnarray*}
It follows immediately that $(\mathcal{M}_\infty,g_\infty)$ splits as the product of a line and a three-dimensional Ricci-flat manifold and hence $g_\infty$ must be flat; this is a contradiction.
\end{proof}

One would like to show that steady GRS limits as in the proposition above are not possible, in which case it would follow for a 4-dimensional shrinking GRS singularity model that $|\func{Rm}|=\func{o}(d^2)$ (this would imply finite topological type).

\begin{remark}
For the exclusion of Ricci-flat ALE limits in the proofs of Propositions \ref{PROP 10} and \ref{PROP 9}, we may alternatively use Theorem \ref{CorExcludeInfinite2} below.
\end{remark}

\section{Excluding an unbounded number of copies of a
Ricci-flat ALE
4-manifold}

\subsection{Statement of the result}
In this section we prove the
result which we applied in the previous section to exclude there existing an unbounded number of disjoint copies of a
Ricci-flat ALE
4-manifold in a closed 4-manifold.

\begin{theorem}\label{CorExcludeInfinite2}
If $\mathcal{N}$ is a closed $4$-manifold and if $\mathcal{A}$ is a
Ricci-flat ALE $4$-manifold, then there can exist at most a bounded number
of disjoint copies of $\mathcal{A}$ embedded in $\mathcal{N}$.
\end{theorem}

\subsection{Spherical space form ends of $4$-manifolds having an unbounded number of disjoint embeddings}

We first prove the following.\footnote{One of the many references for the standard results from algebraic topology used in this section is Hatcher's book \cite{Hat}.}

\begin{proposition}\label{prop: Direct Double}
Let $\mathcal{M}_0$ be a noncompact $4$-manifold without boundary and with a single end diffeomorphic to $S^3/\Gamma \times [0,\infty)$, where $S^3/\Gamma$ is a spherical space form. If there exists a compact $4$-manifold $\mathcal{N}$ containing an unbounded number of disjoint copies of $\mathcal{M}_0$,
then $H_1(S^3/\Gamma,\mathbb{Z} )$ is a direct double, i.e., isomorphic to $A \oplus A$ for some abelian group $A$.
\end{proposition}

Let $\mathcal{M}_0$ be as in the hypotheses of the proposition.
Add $S^3/\Gamma$ to $\mathcal{M}_0$ as its boundary to obtain a compact manifold $\mathcal{M}$ with boundary $\partial \mathcal{M} = S^3/\Gamma \doteqdot \partial $;
so, $\M_0$ is the interior of $\M$.

\begin{lemma}\label{lem: Lemma 1}
Under the hypotheses of Proposition \ref{prop: Direct Double}, the homomorphism $H_m (\pd ; R) \rightarrow H_m (\mathcal{M} ; R)$ induced by $\pd \hookrightarrow \M$ is onto for $m=0,1,2,3$ and any coefficient ring $R$.
The cases $R \cong \ZZ$ and $R \cong \ZZ_p$ will be useful for us. We suppress writing the coefficients, $R$, in the proof.
\end{lemma}

\begin{proof}
This is clear for $m=0$.
Applying Mayer--Vietoris to one copy
of $\M$ and to the complement of its interior
$\overline{\mathcal{N} - \M}$,
which have intersection $\pd$, we obtain the following exact sequence:
\[
\cdots \!\rightarrow H_{m}(\pd)\!\overset{\left(  i_{\ast},j_{\ast}\right)  }{\longrightarrow}\! H_{m}(\M)\oplus H_{m}( \overline{\mathcal{N} - \M})\!\overset{k_{\ast}-l_{\ast}}{\longrightarrow}\!H_{m}(\mathcal{N})\!\overset{\partial_*}{\longrightarrow}\!H_{m-1}(\pd)\rightarrow\!\cdots\!,
\]
where $i:\pd \hookrightarrow \M$, $j:\pd \hookrightarrow \overline{\mathcal{N} - \M}$, $k:\M \hookrightarrow \mathcal{N}$, and $l:\overline{\mathcal{N} - \M} \hookrightarrow \mathcal{N}$ are the inclusion maps.
Suppose that $i_{\ast}: H_{m}(\pd ) \rightarrow H_{m}(\M )$ is not onto.
Then there exists a class $\alpha \in H_{m}(\M )$ which is not in the image of $i_{\ast}$.
Thus, for any such $\alpha$ and any $\beta \in H_{m}(\overline{\mathcal{N} - \M})$,
we have
$k_{\ast} (\alpha) - l_{\ast}(\beta) \neq 0$.

Thus the cokernel of $i_*$, $\operatorname{coker}(i_*) = H_m(\M)/i_* ( H_m ( \pd ) )$, injects into $H_m(\mathcal{N})/l_*(H_{m}(\overline{\mathcal{N} - \M})$
by the map induced by $k_*$.
Indeed, firstly observe that this map is well defined since $k_*(i_*(H_m(\pd)))=l_*(j_*(H_m(\pd))\subset l_*(H_{m}(\overline{\mathcal{N} - \M})$.
Secondly, suppose that there is a nonzero coset $\alpha\, i_* ( H_m ( \pd ) )$ which maps to zero.
Then we have that $\alpha \in H_{m}(\M )$ is not in the image of $i_{\ast}$
and that
$k_*(\alpha) \in  l_*(H_{m}(\overline{\mathcal{N} - \M}) )$, which is a contradiction.

Now suppose that $\mathcal{N}$ contains a finite sequence of disjoint embedded copies of $\M$, denoted by $\{ \M_p\}_{p=1}^I$, for $I$ arbitrarily large.
Let $\M^I=\bigcup_{p=1}^I \M_p$ and
$\pd^I=\pd\M^I=\cup_{p=1}^I\pd \M_p$.
Let $i^I : \pd^I \hookrightarrow \M^I$
and $l^I:\overline{\mathcal{N} - \M^I}\hookrightarrow\mathcal{N}$
denote the inclusion maps.
Since, in the argument above we may replace $\M$ by $\M^I$ (we did not assume connectedness), we have that
$$
\operatorname{coker}(i_{*}^I) = H_m(\M^I)/i_{*}^I ( H_m ( \pd^I ) ) = \oplus_{p=1}^I \operatorname{coker}(i_{p*})
$$
injects into $H_m(\mathcal{N})/l_*^I(H_{m}(\overline{\mathcal{N} - \M^I})$, where $i_p:\pd \M_p\hookrightarrow\M_p$ is the inclusion map.
In this way, we obtain an injection from modules of increasing rank into $H_m(\mathcal{N})/l_*^I(H_{m}(\overline{\mathcal{N} - \M^I})$.
This contradicts the finite dimensionality of $H_m(\mathcal{N})$.

\end{proof}

\begin{remark}
One may easily picture the workings of the preceding formal argument:  If there is any class  $\alpha$ (of any dimension $m$) in
a copy
of $\M$ in $\mathcal{N}$ which does not come from the boundary $\pd$, it is effectively `locked into' that copy of $\M$;
such classes
can neither die in $\mathcal{N}$ nor participate in relations among themselves, because intersecting such null-homologies or homologies (respectively) with $\pd$ would yield a cycle in $\pd$ mapping onto $\alpha$, contrary to assumption.
In other words, the intuition is that the cokernel in each copy of $\M$ is isolated like an island and cannot have any relation with other classes.
\end{remark}

By Lemma \ref{lem: Lemma 1}, the long exact sequence of the pair $(\M,\pd )$ yields the three short exact sequences. Setting $R \cong \ZZ$, we have:
\begin{align}
& 0 \rightarrow H_{2} (\M,\pd ; \ZZ) \rightarrow H_1 (\pd ; \ZZ) \rightarrow H_1 (\M ; \ZZ) \rightarrow 0, \label{eq: Three short exact sequences 1} \\
& 0 \rightarrow H_{3} (\M,\pd ; \ZZ) \rightarrow H_2 (\pd ; \ZZ) \rightarrow H_2 (\M ; \ZZ) \rightarrow 0, \label{eq: Three short exact sequences 2} \\
& 0 \rightarrow H_{4} (\M,\pd ; \ZZ) \rightarrow H_3 (\pd ; \ZZ) \rightarrow H_3 (\M ; \ZZ) \rightarrow 0. \label{eq: Three short exact sequences 3}
\end{align}
Since $\pd$ is a spherical space form, $H_1 (\pd ; \ZZ)$ is torsion, and by \eqref{eq: Three short exact sequences 1} so is $H_1 (\M ; \ZZ)$.
By Poincar\'e duality and the cohomology universal coefficient theorem (see \cite[Theorem 3.2 and Corollary 3.3]{Hat}), we have that
$H_2 (\pd ; \ZZ) \cong H^1 (\pd ; \ZZ) \cong \ZZ^{b_1(\pd)}\oplus \operatorname{Tor} H_0(\pd,\ZZ) \cong 0$, where $b_1(\pd)$ is the first Betti number of $\pd$.
Thus \eqref{eq: Three short exact sequences 2} implies
\begin{equation}\label{eq: H2 M is zero}
H_2 (\M ; \ZZ) \cong 0.
\end{equation}
We also have that $H_3 (\pd ; \ZZ) \cong \ZZ$, but maps to zero in $H_3 (\M ; \ZZ)$ by \eqref{eq: Three short exact sequences 3} since $H_{4} (\M,\pd ; \ZZ) \rightarrow H_3 (\pd ; \ZZ)$ is an isomorphism. Thus $H_3 (\M ; \ZZ) \cong 0$.

By Lemma \ref{lem: Lemma 1} with $R = \ZZ_p$, that is, by
\eqref{eq: Three short exact sequences 1} with coefficients $\ZZ_p$ instead of $\ZZ$, we have
\begin{equation}\label{eq: Three short exact sequences 1 plus}
0 \rightarrow H_2 (\M,\pd ; \ZZ_p) \rightarrow H_1 (\pd ; \ZZ_p) \rightarrow H_1 (\M ; \ZZ_p) \rightarrow 0.
\end{equation}
By Lefschetz duality (see \cite[Theorem 3.43]{Hat}), $H_2 (\M,\pd ; \ZZ_p) \cong H^2 (\M; \ZZ_p)$. By the cohomology universal coefficient theorem,
\begin{equation}\label{eq: the new equation 12}
    0 \rightarrow \operatorname{Ext}_\ZZ^1 (H_1 (\M;\ZZ),\ZZ_p) \rightarrow H^2 (\M; \ZZ_p)  \rightarrow \operatorname{Hom}_\ZZ (H_2 (\M;\ZZ),\ZZ_p) \cong 0 \rightarrow 0.
\end{equation}

As observed below \eqref{eq: Three short exact sequences 3}, $H_1(\M;\mathbb{Z})$ is torsion; thus
$$
H_1 (\M; \ZZ_p) \cong \operatorname{Hom}_{\mathbb{Z}}(\operatorname{Tor}(H_1(\M;\mathbb{Z})),\mathbb{Z}_p).
$$
But this agrees with the usual\footnote{See for example the computations on the Wikipedia page Ext functor or Hatcher’s book \cite{Hat}.
According to these formulae,
$\operatorname{Ext}^1_{\mathbb{Z}}(\mathbb{Z},\mathbb{Z}_p) \cong 0$
and $\operatorname{Ext}^1_{\mathbb{Z}}(\rm{torsion}, \mathbb{Z}_p) \cong \operatorname{Hom} _{\mathbb{Z}}(\rm{torsion}, \mathbb{Z}_p)$.
We use these formulae and the fact that $\operatorname{Ext}$ commutes with direct sums.}
formula which calculates $\operatorname{Ext}^1_{\mathbb{Z}} H_1(\M;\mathbb{Z}),\mathbb{Z}_p) $.
Putting this together with \eqref{eq: the new equation 12}
 yields:
\begin{equation}
H^2 (\M; \ZZ_p) \cong \operatorname{Ext}_\ZZ^1 (H_1 (\M;\ZZ),\ZZ_p)
\cong H_1 (\M;\ZZ_p).
\end{equation}

These observations applied to \eqref{eq: Three short exact sequences 1 plus} yield:
\[
H_1(\pd;\ZZ_p) =  H_1(\M;\ZZ_p) \oplus H_1(\M;\ZZ_p).
\]
To bring this information back towards integer coefficients, apply the homology universal coefficient theorem, for $X=\M$ and $\pd$:
\[
0 \rightarrow H_1(X;\ZZ ) \otimes \ZZ_p \rightarrow
H_1(X;\ZZ_p ) \rightarrow \operatorname{Tor}(H_0(X;\ZZ ), \ZZ_p) \cong 0 \rightarrow 0
\]
to get
\begin{equation}\label{eq: Zp direct double}
H_1(\pd ;\ZZ ) \otimes \ZZ_p \cong H_1(\M ;\ZZ ) \otimes \ZZ_p \oplus H_1(\M ;\ZZ ) \otimes \ZZ_p
\quad \text{for all primes } p.
\end{equation}
This still does not quite prove Proposition \ref{prop: Direct Double}. We need to exclude possibilities like $H_1 (\pd ; \ZZ ) \cong \ZZ_4 \oplus \ZZ_2$ and $H_1(\M ;\ZZ ) \cong \ZZ_2$.
This can be done by returning to line \eqref{eq: Three short exact sequences 1}:
$\operatorname{order}H_1 (\pd ; \ZZ ) =
(\operatorname{order}H_2 (\M , \pd ; \ZZ ))
\times
(\operatorname{order}H_1 (\M ; \ZZ ))$, but
$H_2 (\M , \pd ; \ZZ ) \cong H^2 (\M ; \ZZ )
\cong \operatorname{Ext}^1_{\ZZ} (H_1 (\M;\ZZ ),\ZZ )
\cong \operatorname{torsion}H_1 (\M;\ZZ ) \cong H_1 (\M;\ZZ )$, so $\operatorname{order}H_1 (\pd ; \ZZ )
= (\operatorname{order}H_1 (\M ; \ZZ ))^2$.
Proposition \ref{prop: Direct Double} follows from this, \eqref{eq: Zp direct double}, and the fundamental theorem of finite abelian groups.
$\Box$

\subsection{Proof of Theorem \ref{CorExcludeInfinite2}}
We may now prove Theorem \ref{CorExcludeInfinite2}
using Bamler's argument in \S 6 of \cite{Zhang}.
For a contradiction, suppose that there exists
a nonflat Ricci-flat ALE $4$-manifold $\mathcal{A}$ that admits unbounded collections of disjoint embeddings in a compact $4$-manifold $\mathcal{N}$.
By Proposition \ref{prop: Direct Double}, the spherical space form $S^3/\Gamma$, $\Gamma$ nontrivial, of the end of $\mathcal{A}$ satisfies $H_1(S^3/\Gamma,\mathbb{Z} ) \cong A \oplus A$ for some abelian group $A$.
This implies (see e.g. Lemma 6.3 in \cite{Zhang}) that either $\Gamma$ is the binary dihedral group $D_n^*$ with $n$ even or $\Gamma$ is the binary icosahedral group of order $120$.
Now, by Lemma 6.5 in \cite{Zhang}, $b_2(\mathcal{A}) \geq 1$, which  contradicts \eqref{eq: H2 M is zero}. $\Box$

\begin{remark}
We may also rule out the case where $\Gamma$ is the binary icosahedral group $I^*$ of order $120$, that is, where $\pd = S^3/I^*$ is the Poincar\'e homology sphere, as follows. By Rochlin's theorem, if $\pd \M = S^3/I^*$, then $H_*(\M;\ZZ) \not \cong H_*(\operatorname{pt};\ZZ)$. As in Lemma \ref{lem: Lemma 1}, the facts that $H_*(S^3/I^*;\ZZ) \cong H_*(S^3;\ZZ)$
and $H_*(\M;\ZZ) \not \cong H_*(\operatorname{pt};\ZZ)$ imply, via Mayer--Vietoris, that $H_*(\mathcal{N};\ZZ)$ is infinitely generated, a contradiction.
\end{remark}

\textbf{Acknowledgment.} We would like to thank the referee for a number of helpful suggestions.

\end{document}